\documentclass[11pt]{amsart}

\usepackage{amsmath}
\usepackage{amsfonts}
\usepackage{amssymb}
\usepackage{amsthm}
\usepackage{amsrefs}
\usepackage{graphicx}
\usepackage{pb-diagram}
\usepackage{epstopdf}
\usepackage{amscd}
\usepackage{pdfpages}
\usepackage{bbm}
\usepackage{multirow}
\usepackage[mathscr]{eucal}
\usepackage{epsfig,epsf,epic}
\usepackage{pdfsync}
\usepackage{enumerate}
\usepackage{etex}
\usepackage[all]{xy}
\usepackage{fancyref}
\usepackage{mathtools}
\usepackage{scalerel,stackengine}
\usepackage{comment}
\usepackage{hyperref}

\stackMath

\newtheorem{theorem}{Theorem}[section]

\newtheorem{definition}[theorem]{Definition}

\newtheorem{lemma}[theorem]{Lemma}
\newtheorem{remark}[theorem]{Remark}

\numberwithin{equation}{section}

\DeclareMathOperator{\trace}{tr}

\newcommand{\real}{\mathbb{R}}
\newcommand{\comp}{\mathbb{C}}

\newcommand{\inte}{\mathbb{Z}}

\newcommand{\dd}[1]{\frac{\partial}{\partial #1}}

\newcommand{\half}{\frac{1}{2}}

\newcommand{\pdpd}[3]{\frac{\partial^2 #1}{\partial #2\partial #3}}

\newcommand{\liealg}{\mathfrak{g}}
\newcommand{\lietorus}{\mathfrak{t}}
\newcommand{\subalg}{\mathfrak{h}}
\newcommand{\adjact}{\text{Ad}}
\newcommand{\moment}{\mu}
\newcommand{\rootspace}{\mathcal{R}}
\newcommand{\weylchamb}{\mathcal{C}_+}
\newcommand{\weightlattice}{\Lambda}
\newcommand{\gqrep}{\beta}

\begin{document}
\title[Toeplitz Quantization of coadjoint orbits]{Geometric quantization and quantum moment maps on coadjoint orbits and K\"ahler-Einstein manifolds}	
	
\author[Leung]{Naichung Conan Leung}
\address{The Institute of Mathematical Sciences and Department of Mathematics\\ The Chinese University of Hong Kong\\ Shatin \\ Hong Kong}
\email{leung@math.cuhk.edu.hk}
	
\author[Li]{Qin Li}
\address{Southern University of Science and Technology, Shenzhen, China}
\email{liqin@sustech.edu.cn}

\author[Ma]{Ziming Nikolas Ma}
\address{The Institute of Mathematical Sciences and Department of Mathematics\\ The Chinese University of Hong Kong\\ Shatin \\ Hong Kong}
\email{nikolasming@outlook.com}

\begin{abstract}
Deformation quantization and geometric quantization on K\"ahler manifolds give the mathematical description of the algebra of quantum observables and the Hilbert spaces respectively, where the later forms a  representation of quantum observables asymptotically via Toeplitz operators. When there is a Hamiltonian $G$-action on a K\"ahler manifold, there are associated symmetries on both the quantum algebra and representation aspects. We show that in nice cases of coadjoint orbits and K\"ahler-Einstein manifolds, these symmetries are strictly compatible (not only asymptotically). 
\end{abstract}

\maketitle

\section{Introduction}\label{sec:introduction}

Quantizing a classical mechanical system on $\left( X=T^{\ast }\mathbb{R}^{n},\omega =\sum dx^{j}\wedge dp_{j}\right) $ is amount to treating position $x^{j}$ and moment $p_{j}$ as operators $\hat{x}^{j}=x^{j}\cdot $ and $\hat{p}_{j}=i\hslash \frac{\partial }{\partial x^{j}}$ on $L^{2}\left( \mathbb{R}^{n}\right) $ respectively, thus realizing the uncertainty principle $\left[ \hat{x}^{j},\hat{p}_{k}\right] =i\hslash\delta _{k}^{j}$. In fact, this gives an action of $C^{\infty }\left( X\right)[[\hslash]] $ on $L^{2}\left( \mathbb{R}^{n}\right) $ with a non-commutative algebra structure $\star _{\hslash}$ on $C^{\infty }\left( X\right) \left[ \left[ \hslash\right] \right] $. $L^{2}\left( 
\mathbb{R}^{n}\right) $ can also be realized as $\Gamma _{L^{2}}\left(X,L\right) \cap Ker\left( \nabla_{\mathcal{P}_{\mathbb{R}}}\right) $, the space of $L^{2}$-sections of the trivial line bundle $L$, equipped with a unitary connection $\nabla$ with curvature $-i\omega $, over $X$ which are constant along the real polarization $\mathcal{P}_{\mathbb{R}}=\left\langle \frac{\partial }{\partial p_{j}}\text{'s}\right\rangle $, namely those sections which are independent of $p_{j}$'s.

Quantizing a physical system should be independent of the choice of polarizations. If we use the complex polarization $\mathcal{P}_{\mathbb{C}}=\left\langle \frac{\partial }{\partial \bar{z}_{j}}\text{'s}\right\rangle $ on $X=\mathbb{C}^{n}$ with $z_{j}=x^{j}+ip_{j}$, then 
\[
\Gamma _{L^{2}}\left( X,L\right) \cap Ker\left( \nabla_{\mathcal{P}_{\mathbb{C}%
}}\right)=\Gamma _{L^{2}}\left( X,L\right) \cap Ker(\bar{\partial}) =H_{L^{2}}^{0}\left( X,L\right) 
\]
is the space of $L^{2}$-holomorphic sections of the trivial line bundle $L$. This is a representation of polynomials on $X=\mathbb{C}^n$ under the Wick product.  Explicitly, $z_j$ and $\bar{z}_j$ act as operators $z_j \cdot$ and $\hslash \dd{z_j}$ respectively. This action can also be obtained via the Toeplitz operators 
\begin{align*}
	T : C^{\infty }\left( X\right) &\rightarrow End\left( H_{L^{2}}^{0}\left(
	X,L\right) \right)  \\
	f &\mapsto T_{f}=\Pi \circ M_{f},
\end{align*}%
which is given by the multiplication $M_f$ followed by the orthogonal projection $\Pi :\Gamma _{L^{2}}\left( X,L\right) \rightarrow H_{L^{2}}^{0}\left(X,L\right)$ with respect to the volume form $i^n\cdot e^{-|z|^2/\hslash}dz^1d\bar{z}^1\cdots dz^nd\bar{z}^n$ on $\mathbb{C}^n$. By turning $\hslash$ to a formal variable, we obtain the star product for smooth functions on $\mathbb{C}^n$.

This construction can be generalized to quantization of K\"ahler manifolds. In the rest of this paper, we will let $\left( X,\omega,J \right)$ denote a compact K\"ahler manifold with prequantum line bundle $(L,\nabla)$ satisfying $\nabla^2=-i\omega$. The Hilbert spaces are holomorphic sections $H^0(X,L^m)$ of positive tensor powers of $L$. 
Given $f \in C^{\infty}(X)$, one can take multiplication operator $M_{f} : L^2(X,L^m) \rightarrow L^2(X,L^m)$ giving a representation of commutative algebra $C^{\infty}(X)$ on $L^2(X,L^m)$. To obtain an action on $H^0(X,L^m)$ one takes the projections to define Toeplitz operator $T^{(m)} : C^{\infty}(X) \rightarrow \text{End}(H^0(X,L^m))$ as
\begin{equation}\label{eqn:toeplitz_definition}
T^{(m)}_f:= T^{(m)}(f) := \Pi \circ M_f,
\end{equation}
where $\Pi : L^2(X,L^m) \rightarrow H^0(X,L^m)$ denotes the orthogonal projection to holomorphic sections. The Toeplitz operators is closed under composition only in the asymptotic sense as $\hslash = \frac{1}{m} \rightarrow 0$.
By the result of \cite{bordemann1994toeplitz}, there is a sequence of bidifferential operators $C_l : C^{\infty}(X) \times C^{\infty}(X) \rightarrow C^{\infty}(X)$ such that for any $f ,g \in C^{\infty}(X)$ there is an asymptotic expansion as $m \rightarrow \infty$
\begin{equation}
T^{(m)}_f \circ T^{(m)}_g \sim \sum_{l \geq 0} T^{(m)}_{C_l(f,g)} m^{-l},
\end{equation}
with $C_{0} = f\cdot g$ the ordinary product and $C_1(f,g)-C_1(g,f) = \{f,g\}$ the Poisson bracket (\cite{schlichenmaier2010berezin}*{Theorem 4.5.}). Furthermore, these operators $C_l$'s can be put together into an associative $\star$-product via the formula
\begin{equation}
f \star g := \sum_{l}  \hslash^{l}\cdot C_l(f,g),
\end{equation}
as a deformation quantization $(C^{\infty}(X)[[\hslash]],\star)$ which is known as the Berezin-Toeplitz star product. The non-commutative algebra $\left( C^{\infty }\left( X\right)[[\hslash]] ,\star\right) $, the Hilbert space $\mathcal{H}_m=H^0(X,L^m) $ and the operators $T_f^{(m)}$'s are called the deformation quantization, geometric quantization and Berezin-Toeplitz quantization respectively.

%There are two possible mathematical formulation of quantization, namely deformation quantization yielding non-commutative deformation $C^{\infty}(X)[[\hslash]]$ of ring of smooth function $C^{\infty}(X)$, and geometric quantization yielding vector spaces $H^0(X,L^m)$. $C^{\infty}(X)[[\hslash]]$ does not act directly on $H^0(X,L^m)$'s as a representation but only in the asymptotic sense as $\hslash = \frac{1}{m} \rightarrow 0$, and this is given by the Toeplitz operator $T^{(m)}$'s on $H^0(X,L^m)$'s.

The quantization of $X$ is a much more complicated, but important question, especially in representation theory. As when $\left( X,\omega,J \right)$ has a Hamiltonian $G$-symmetry with moment map $\mu :X\rightarrow \liealg^{\ast }\text{,}$ then its geometric quantization $\mathcal{H}:=H^0(X,L)$ would be a $G$-representation $G\rightarrow GL\left( \mathcal{H}\right)$. This representation is itself very important, a closely related question is the Guillemin-Sternberg conjecture on symmetry commutes with geometric quantization, which there has been extensive studies, e.g. \cites{Guillemin-Sternberg,Ma-Zhang,Tian-Zhang,Vergne}. The associated Lie algebra representation is denoted as 
\begin{equation}\label{equation: G-representation-Hilbert-space}
\gqrep : \liealg \rightarrow gl(\mathcal{H}).
\end{equation}
And one expects most interesting representations should come from such a
quantization process. 

By pulling back linear functions on $\liealg$ via $\mu $, we have%
\[
\mu ^{\ast }:\liealg \rightarrow C^{\infty }\left( X\right) \text{.}
\]%
Composing with the Toeplitz operator  $T:C^{\infty }\left( X\right)
\rightarrow gl\left( \mathcal{H}\right) $, we obtain a map%
\[
-i T\circ \mu ^{\ast }:\liealg\rightarrow gl\left( \mathcal{H}\right) \text{.}
\]%
It is natural to compare this with the representation $ \gqrep$. It turns out that they do not agree with each other even in the case of $X$ being coadjoint orbits, unless we replace $\moment$ by the Karabegov moment map $\tilde{\moment}$. 

\begin{definition}\label{def:karabegov_moment_map}
	A map $\tilde{\moment} :X \rightarrow \liealg^*$  is called a Karabegov moment map if it is the moment map associated to the $2$-form $\tilde{\omega}$ given by 
	\begin{equation}\label{eqn:defining_karabegov_form}
	\tilde{\omega}:= \omega + i\cdot\text{Ric}_X,
	\end{equation} 
	i.e. $\tilde{\moment}$ is $G$-equivariant and satisfies $\iota_{v} \tilde{\omega} = d\tilde{\moment}^*(v)$ for all $v \in \liealg$\footnote{We use the same notation for an element in $\mathfrak{g}$ and its associated vector field on $X$.}. 
\end{definition}

The main result of this paper is that in nice cases such as coadjoint orbits or Hamiltonian K\"ahler-Einstein manifolds the representation $\gqrep$ can be obtained via the composition of the Toeplitz quantization $T = T^{(1)}$ with Karabegov moment map $\tilde{\moment}$.

\begin{theorem}[= Theorem \ref{thm:coadjoint_orbit}]
	If $X=O_{\xi }$ is an integral coadjoint orbit of $G$, then 
	\[
	-i T\circ \tilde{\moment}^{\ast }=\gqrep.
	\]%
\end{theorem}

We show that this is also true for any K\"ahler-Einstein manifolds with $G$-symmetry.

\begin{theorem}[=Theorem \ref{thm:general_kahler_einstein_thm}]
	If $X$ is a compact K\"ahler-Einstein manifold with $G$-symmetry, then 
	\[
	-i T\circ \tilde{\moment}^{\ast }=\gqrep.
	\]
\end{theorem}

\begin{remark}
	The above assumption implies that the Einstein constant is positive and the $G$-symmetry is Hamiltonian. 
\end{remark}

Our main theorems say the following diagram commutes:
\begin{equation}\label{equation: commutative-diagram}
\xymatrix
{ & C^\infty(X) \ar[dr]^{T}  & \\
	\mathfrak{g}\hspace{1.5mm} \ar[ur]^{-i\cdot\tilde{\mu}^*}\ar[rr]_{\beta} & & gl(\mathcal{H}),
}
\end{equation}
which means the compatibility between the quantum symmetries. The map $\beta$ describes the infinitesimal quantum symmetry on $\mathcal{H}$. The left arrow describes the ``quantum symmetry'' on the algebra of observables. The quantum moment map $ \tilde{\moment}^*$ has close relation with the Fedosov quantization scheme will be explained in details in \S \ref{sec:quantization}.

%As a quantization scheme on general prequantizable K\"ahler manifolds, the Toeplitz operators induce the Berezin-Toeplitz star products: Let $\hslash = \frac{1}{ m}$ and consider the $m$-th power $L^m$, with corresponding Hilbert space $\mathcal{H}^{(m)}=H^0_{L^2}(X,L^{\otimes m})$ and Toeplitz operator $T^{(m)}$. There is the asymptotic expansion as $m \rightarrow \infty$: 
%$$T^{(m)}_f \circ T^{(m)}_g \sim T^{(m)}_{f \star_{\hslash} g},$$
%which uniquely determines $\star_{\hslash}$. The Karabegov form defined in \cite{karabegov2000identification}, classifying star products, is given by $-\frac{1}{\hslash} \omega+ iF_{K}$ for $\star_{\hslash}$. We may consider the quantum moment map $\tilde{\moment}$ for the Karabegov form $\frac{1}{\hslash} \omega - iF_{K}$ as a modification of $\moment$ in the presence of Hamiltonian $G$-symmetry. In general, the above diagram only commute in the asymptotic sense as $\hslash \rightarrow \infty$. See Section \ref{subsection: quantum-moment-map} and Remark \ref{rem:asymptotic_communte} for a more detailed explanation. 

\section*{Acknowledgement}
N. C. Leung was supported by grants of the Hong Kong Research Grants Council (Project No. CUHK14301117 \& CUHK14303518) and direct grants from CUHK.
Q. Li was supported by Guangdong Basic and Applied Basic Research Foundation (Project No. 2020A1515011220) and National Science Foundation of China (Project No. 12071204).

\section{Main theorems}\label{sec:main_theorem}
In this section, we prove our theorems by differential geometric and Lie theoretic computations.

\begin{comment}
Let $G$ be a compact connected Lie-group with Lie-algebra $\liealg$, and denote $\langle \cdot,\cdot\rangle$ be the natural pairing between $\liealg$ and its dual $\liealg^*$. Suppose the group $G$ acts on $X$ preserving K\"ahler structure, equipped with a moment map $\moment : X \rightarrow \liealg^*$. We further assume that $(L,\nabla)$ is a $G$-equivariant line bundle with $F_{L} = -i\omega$. 
\end{comment}

We take the following convention: For every smooth function $f$ on $X$, the Hamiltonian vector field $X_f$ is defined by $\iota_{X_f}(\omega)=df$, and there is
$$
\{f,g\}:=\omega(X_g,X_f).
$$
This convention guarantees that the map $f\mapsto X_f$ is a Lie algebra homomorphism:
$$
X_{\{f,g\}}=[X_f,X_g].
$$

\subsection{Hamiltonian K\"ahler-Einstein manifold}

When $(X,\omega,J)$ is K\"ahler-Einstein with $i\text{Ric}_X =\lambda \omega$, Karabegov moment map is given by $\tilde{\moment} = (1 +\lambda) \moment$.

\begin{lemma}[Tuynman \cite{tuynman1987quantization}, see also \cite{bordemann1991gl}*{Proposition 4.1.}]\label{lem:tuynman_formula}
	Let $(X,\omega,J,L)$ be K\"ahler manifold with a pre-quantium line bundle, then we have 
	$$
	\Pi \circ \nabla_{X_f} = \frac{i}{2}\Pi \circ T_{\Delta(f)}
	$$
	acting on holomorphic sections $H^0(X,L)$, where $\Delta$ is the Laplacian with respect to the K\"ahler metric.
\end{lemma}

\begin{theorem}\label{thm:general_kahler_einstein_thm}
	Let $(X,\omega,J,L)$ be K\"ahler manifold with a pre-quantium line bundle together with a $G$-symmetry. Suppose further that $X$ is K\"ahler-Einstein, then we have
	$$
	\gqrep =  -i T \circ \tilde{\moment}^*.
	$$ 
\end{theorem}

\begin{proof}
	For any $w \in \liealg$,  let $f_{w}:= \langle \moment , w\rangle$ denote the pullback of $w$ via the moment map $\mu$. $J X_{f_{w}} = \nabla f_{w}$ is a holomorphic Killing vector field on $(X,\omega,J)$, and using the theorem by Matsushima (see e.g. \cite{ballmann2006lectures}*{Theorem 6.16.}) there is a bijection between eigenspace $E_{2\lambda}$ of $\Delta$ and space of holomorphic Killing vector fields by taking their gradient vector fields. That is $\Delta f_{w} = -2 \lambda f_{w}$. 
	
	Notice $\langle \tilde{\moment}, w \rangle  = (1+ \lambda) f_{w}$, and therefore using the Lemma \ref{lem:tuynman_formula} we have
	$$
	-i T \circ \tilde{\moment}^* = -i T \big( f_{w} -\half  \Delta (f_{w})   \big) =Q_f,
	$$
		where $Q_{f} = \Pi \circ (\nabla_{X_f} -i  f)$. Finally we have $[Q_{f_{w}} , Q_{f_{v}} ] = Q_{\{f_{w},f_{v}\}} = Q_{f_{[w,v]}}$ due to the fact that the vector field $X_{f_{v}}$ comes from the $G$-action which preserves the K\"ahler structure $(\omega,J)$, and therefore we obtain the representation $\gqrep$ of $\liealg$ on $H^0(X,L)$ by sending $w \mapsto Q_{f_{w}}$. 
	\end{proof}

\subsection{Coadjoint orbits}\label{sec:recall_coadjoint_orbit}
In this section, we recall the notions and results for coadjoint orbits that are necessary for the proof of our second main theorem. 

\subsubsection{Coadjoint orbits as symplectic manifolds}
We consider a compact connected Lie group $G$ with a fixed choice of maximal torus $T$. Without loss of generality, we assume that $\liealg$ has no Abelian factor. We let $\liealg^*$ be the dual of $\liealg$ equipped with coadjoint action $\adjact^*$. We let the non-degenerated negative definite Killing form on $\liealg$ to be $\kappa(\cdot , \cdot)$ which identifies $\liealg$ and $\liealg^*$, and we will abuse our notation by treating $\kappa$ as a pairing on $\liealg^*$ as well. 

We denote $X := \mathcal{O}_{\xi}$ the coadjoint orbit through $\xi \in \liealg^*$, which can be identified with $G / G_{\xi}$ where $G_{\xi}$ denotes the stabilizer subgroup of $\xi$ with Lie subalgebra $\liealg_{\xi}$. $X$ is equipped with the natural Kirillov-Kostant-Souriau symplectic form (KKS form in short) given by the formula $\omega_{\xi}(\eta_1,\eta_2):= \langle \xi, [\eta_1,\eta_2] \rangle$, for $\eta_1,\eta_2 \in \liealg/\liealg_{\xi} \cong T_{o}(G / G_{\xi})$ at the identity $o \in G/G_{\xi}$, where $\langle \cdot,\cdot \rangle$ is the natural pairing between $\liealg$ and $\liealg^*$. The natural $G$-action $G \curvearrowright X = G / G_{\xi}$ is Hamiltonian and its moment map is given by the natural embedding $\moment_{\xi} : \mathcal{O}_{\xi} \hookrightarrow \liealg^*$. 

\subsubsection{Root space decomposition}\label{sec:root_space}
We let $\liealg_{\comp} = \liealg \otimes_{\real} \comp$ denote the complexification of $\liealg$ and similarly for $\subalg_{\comp} = \subalg\otimes_{\real} \comp$ for any subalgebra $\subalg \subset \liealg$. We let $\rootspace$ to be the set of complex roots $\alpha \in i \lietorus^* \subset \lietorus^*_{\comp}$, where $\lietorus$ is the Lie-algebra of $T$. We further choose a set of simple roots $\mathcal{S} = \{ \alpha_1,\dots,\alpha_r\}$ which specifies a decomposition $\rootspace = \rootspace_+ \sqcup \rootspace_-$ into positive and negative roots, as well as a choice of fundamental Weyl chamber $\weylchamb$. We have the root space decomposition $\liealg_{\comp} = \lietorus_{\comp} \oplus \bigoplus_{\alpha \in \rootspace_+} (\liealg_{\alpha} \oplus \liealg_{-\alpha})$. Given a root $\alpha$, we define the co-root $H_{\alpha} \in \lietorus_{\comp}$ to $\alpha$ by taking $\kappa(H_{\alpha},\cdot):= \frac{2\alpha}{\kappa(\alpha,\alpha)}$, which is the unique element in $[\liealg_{\alpha},\liealg_{-\alpha}]$ satisfying $\alpha(H_{\alpha}) = 2$. 

For computations on the homogeneous space $X = G/G_{\xi}$, we further introduce the notation $\rootspace_{\xi}:= \{ \alpha \in \rootspace \ | \ \kappa(\xi,\alpha) = 0 \}$ and the set of complementary root $\rootspace_{\xi}^{c} := \rootspace \setminus \rootspace_{\xi}$ as in \cite{arvanitogeorgos2003introduction}, and similarly $\rootspace_{\xi,+}  = \rootspace_{\xi} \cap \rootspace_+$ and $\rootspace_{\xi,+}^{c} := \rootspace_+ \setminus \rootspace_{\xi,+}$. Notice that $\rootspace_{\xi}$ and $\rootspace_{\xi,+}$ only depends on the smallest closed strata containing $\xi$ in the fundamental Weyl chamber $\weylchamb$ from its definition.

\subsubsection{K\"ahler structures on coadjoint orbits}
For any $\xi\in\liealg^*$, letting $\xi^{\vee} \in \liealg$ be the corresponding element via the identification by $\kappa$, we have $\liealg_{\xi} = \lietorus \oplus \bigoplus_{\alpha \in \rootspace_{\xi,+}} (\liealg_{\alpha} \oplus \liealg_{-\alpha}) \cap \overline{(\liealg_{\alpha} \oplus \liealg_{-\alpha})}$. Therefore, we have the natural identification $\liealg / \liealg_{\xi} = \liealg_{\xi}^{c}:= \bigoplus_{ \alpha \in \rootspace_{\xi,+}^{c}}(\liealg_{\alpha} \oplus \liealg_{-\alpha}) \cap \overline{(\liealg_{\alpha} \oplus \liealg_{-\alpha})}$, together with a direct sum decomposition $\liealg = \liealg_{\xi} \oplus \liealg_{\xi}^c$. There is a unique  $G$-invariant integrable complex structure $J$ on $\mathcal{O}_{\xi}$ with 
$$
T^{1,0}_e (G / G_{\xi}) = \bigoplus_{\alpha \in \rootspace_{\xi,+}^{c}} \liealg_{\alpha}, \ \ \ \  T^{0,1}_e (G / G_{\xi}) = \bigoplus_{\alpha \in \rootspace_{\xi,+}^{c}} \liealg_{-\alpha}.
$$
This $J$ together with $\omega_{\xi}$ gives a $G$-invariant K\"ahler structure on $X = \mathcal{O}_{\xi}$. We write $g_\xi$ for the metric tensor associated to this K\"ahler structure.  

\begin{comment}
\subsubsection{Weight lattice and highest weight theory}\label{sec:weight_lattice}
Given a root $\alpha$, we define the co-root $H_{\alpha} \in \lietorus_{\comp}$ to $\alpha$ by taking $K(H_{\alpha},\cdot):= \frac{2\alpha}{K(\alpha,\alpha)}$, which is the unique element in $[\liealg_{\alpha},\liealg_{-\alpha}]$ such that $\alpha(H_{\alpha}) = 2$. We define the set of fundamental weights $\{w_1,\dots,w_r\}$ to be the basis of $\lietorus_{\comp}^*$ satisfying the equation $w_i(H_{\alpha_j}) = \delta_{ij}$. We call the lattice spanned by $w_i$'s the weight lattice $\weightlattice$. Therefore we can always write $\beta = \sum_{i} \beta(H_{\alpha_i}) w_i$ for a general weight $\beta \in \weightlattice$. Those $\beta$ lying in the fundamental Weyl chamber $\weylchamb$, or equivalently satisfying $\beta(H_{\alpha_i}) \geq 0$ for simple roots $\alpha_i$'s, are called domaniant weight.

For any irreducible complex representation $V$ of $G$, its restriction to $T$ gives a decomposition $V = \bigoplus_{\lambda \in \weightlattice(V)} V_{\lambda}$ in to irreducible representation $V_{\lambda}$ of $T$, called the weight space decomposition. $\lambda \in \weightlattice(V)$ is called the highest weight of $V$ if $\liealg_{\alpha} \cdot V_{\lambda} = \{0\}$ for all positive root $\alpha \in \rootspace_+$. The famous highest weight theorem says that for any given domaniant weight $\lambda$, there is an unique irreducible representation of $G$ with highest weight $\lambda$. 
\end{comment}

\subsubsection{Pre-quantum line bundles on coadjoint orbits} 
Without loss of generality, we take $\xi \in i\weylchamb$ with integral symplectic class $[\omega_{\xi}] \in H^2(X,2\pi \inte)$. This allows us to define the associated unitary line bundle $L_{\xi}$ on $X = G / G_{\xi}$ equipped with a unitary connection $\nabla^{\xi}$ on $L_{\xi}$ such that the curvature form $(\nabla^{\xi})^2 =-i \omega_{\xi}$. This gives a pre-quantum line bundle $(L_{\xi},\nabla^{\xi})$ on $X$ with a natural $G$-equivariant structure given by left translations. 

\begin{comment}
There exist a unique character $\tau_{\xi} : G_{\xi} \rightarrow U(1)$ such that 
\begin{equation}
\tau(e^{X}) = e^{-i \langle \xi, X \rangle}
\end{equation}
for $X \in \liealg_{\xi}$. 
This allows us to define the associated unitary line bundle $L_{\xi}$ on $X = G / G_{\xi}$ by taking the space of smooth section to be 
\begin{equation}\label{eqn:defining_pre_quantum_line_bundle}
\Gamma(X,L_{\xi}):= \{ s \  | \ s : G \rightarrow \comp, \ s(g\cdot h^{-1}) = \tau(h) \cdot s(g)\},
\end{equation} 
for $g \in G$ and $h \in G_{\xi}$. It is equipped a natural $G$-action given by $(g\cdot s)(g'):= s(g^{-1}\cdot g')$ which is compatible with the natural $G$-action on $X$. Furthermore, it is naturally equipped with a unitary connection $\nabla^{\xi}$ on $L_{\xi}$ such that the curvature form $(\nabla^{\xi})^2 =-i \omega_{\xi}$. This gives a pre-quantum line bundle $(L_{\xi},\nabla^{\xi})$ on $X$. 
\end{comment}

\subsection{Toeplitz quantization on coadjoint orbits}\label{sec:toeplitz_on_coadjoint_orbit}
Using the results from \cites{bordemann1986homogeneous, borel1958characteristic, borel1959characteristic, borel1960characteristic}, the Ricci form $\text{Ric}_{\xi} = -iF_{K}$ of $X$ can be computed and is given by 
\begin{equation}\label{eqn:ricci_2_form}
\text{Ric}_{\xi}(X,Y) = \langle \delta_{2\xi},[X,Y] \rangle,
\end{equation} 
for $X,Y \in \liealg/\liealg_{\xi} \cong T_{o}(G / G_{\xi})$, where $\delta_{\xi} = \sum_{\alpha \in \rootspace_{\xi}^{c} } \half \alpha$. 
As a consequence, we find that $X$ is K\"ahler-Einstein with Einstein constant $\lambda$ exactly when $\xi = \frac{2}{\lambda}\delta_{\xi}$. For $X$ equipped with other K\"ahler metrics, it would be interesting to ask whether the identity in Theorem \ref{thm:general_kahler_einstein_thm} still hold for these metrics. 

In the case of coadjoint orbit $X = G/G_{\xi}$, the Karabegov form in equation \eqref{eqn:defining_karabegov_form} is explicitly given by $\tilde{\omega}_{\xi} =  \omega_{\xi} + \omega_{2\delta_{\xi}}$. The Karabegov moment map is given by $\tilde{\moment}_{\xi} = \moment_{\xi} + \moment_{2\delta_\xi}$, where $\moment_{\xi}$ and $\moment_{2\delta_\xi}$ are moment map on $G/G_{\xi}$ with respect to symplectic form $\omega_{\xi}$ and $\omega_{2\delta_{\xi}}$ respectively.

\subsubsection{Peter-Weyl theorem}\label{sec:Peter_weyl_theorem}

%We consider the Laplace-Beltrami operator $\Delta = dd^* + d^*d$ acting on smooth function $\mathcal{C}^{\infty}(X)$, giving the spectral decomposition $L^{2}(X) = \overline{\bigoplus_{i=1}^{\infty} E_{\lambda_i}}$ where $E_{\lambda_i}$ is the eigenspace associated to eigenvalue $\lambda_i$. We notice that each $E_{\lambda_i}$ is a finite dimensional unitary representation of $G$. We are interested in studying the possible $G$-equivariant morphism $\liealg \hookrightarrow E_{\lambda_i}$ as an irreducible component. 

Using the Peter-Weyl theorem, we have a decomposition of the space of complex valued square integrable functions 
\begin{equation*}
L^2(G) = \overline{\sum_{\rho \in \hat{G}} V^*_{\rho} \otimes V_\rho}
\end{equation*}
as $G \times G$ modules, where $\hat{G}$ is the set of irreducible representations of $G$. Explicitly, an element $w \otimes v$ is treated as a function given by $f_{w,v}(g):= \langle w, \rho(g)(v) \rangle$, where $\langle \cdot,\cdot \rangle$ is the natural pairing between $V_\rho^*$ and $V_{\rho}$. Taking the right $G_{\xi}$-invariant part we obtain the corresponding decomposition
\begin{equation}
L^2(X) = \overline{\sum_{\rho \in \hat{G}} V^*_{\rho} \otimes V_\rho^{G_{\xi}}}
\end{equation} 
as left $G$-modules, where $V_\rho^{G_{\xi}}$ refers to the subspace fixed by $G_\xi$. We consider the Laplace-Beltrami operator $\Delta = dd^* + d^*d$ acting on smooth function $C^{\infty}(X)$, which is invariant under the left $G$-action and hence acting on the individual component $V^*_{\rho} \otimes V_\rho^{G_{\xi}}$. This action is computed explicitly in \cite{yamaguchi1979spectra} for the K\"ahler-Einstein metric, we modify its proof for the K\"ahler metric $g_{\xi}$. We take a basis $e_1,\dots,e_l$ of $\liealg_{\xi}^c \cong T_o(G/G_{\xi})$, and obtain a local coordinate near $[g] \in G/G_{\xi}$ by the following map:
\begin{equation}\label{eqn:local_exponential_coordinates}
(x_1,\dots,x_l) \mapsto [g\cdot\exp(x_1e_1 + \cdots + x_l e_l)]
\end{equation}
for every $g \in G$. Here $\exp$ denotes the Lie-theoretic exponential map, and notice that this may not coincide with the Riemannian exponential map because the K\"ahler metric $g_{\xi}$ is not naturally reductive metric in general.

\begin{lemma}\label{lem:explicit_laplacian}
        There exists a linear map $\Omega:V_\rho\rightarrow V_\rho$, such that for the function $f_{w,v}$ on $G/G_\xi$, we have $\Delta f_{w,v} = f_{w,\Omega\cdot v}$. Explicitly, $\Omega$ is given by
	$$
	\Omega (v) = \sum_{i,j=1}^l a^{ij} d\rho(e_i) \circ d\rho(e_j) (v),
	$$
	where $a_{ij}:= g_{\xi,o}(e_i,e_j)$'s are the matrix coefficients associated to metric $g_{\xi,o}$ on $\liealg_{\xi}^c$ at identity $o \in G/G_{\xi}$, and $a^{ij}$'s are their inverse matrix coefficients. Here $d\rho$ denotes the Lie algebra representation associated to $\rho$.
	\end{lemma}

\begin{proof}
	With the identification $dl_g: \liealg_{\xi}^{c} \cong T_{g} (G/G_{\xi})$ via left translation $l_g : X \rightarrow X$, we use the exponential coordinates from equation \eqref{eqn:local_exponential_coordinates}, and compute $\big( \nabla_{\dd{x_i}} \dd{x_j} \big)|_{g} \in \liealg_{\xi}^{c}$ for arbitrary $i,j$. Using the Baker-Campbell-Hausdorff formula we notice that we can write $g \cdot \exp(x_i e_i + x_j e_j) = g \cdot \exp(x_j e_j - \frac{x_i x_j}{2}[e_j,e_i] + o(\|x\|^2)) \exp(x_i e_i)$, where $\| x\|$ refers to the norm on the tangent space $T_{o} (G/G_{\xi})$ given by $g_{\xi,o}$. We have 
	$$
 	dl_{g^{-1}}\big( \dd{x_j}\Bigg|_{g \exp(x_i e_i)} \big) = \big( e_j - \frac{x_i }{2}[e_j,e_i] + o(\|x\|) \big)^{\#}|_{\exp(x_i e_i)} ,
	$$
	where $X^{\#}$ refers to the vector field generated by the left action for $X \in \liealg$. We have 
	$$dl_{g^{-1}}\big( (\nabla_{\dd{x_i}} \dd{x_j})|_{g} \big) = \big(\nabla_{e_i} ( e_j - \frac{x_i }{2}[e_j,e_i] )^{\#}\big)_o = \big(\nabla_{e_i} e_j^{\#} \big)_o + \half [e_i,e_j]^{\#}$$
	 at $o \in G/G_{\xi}$. Making use of the formula \cite{arvanitogeorgos2003introduction}*{Proposition 5.2.} we have that \begin{equation}\label{eqn:lemma_calcuation}
	 dl_{g^{-1}}\big( (\nabla_{\dd{x_i}} \dd{x_j})|_{g} \big)  = U(e_i,e_j),
	 \end{equation} 
	 where $U(\cdot,\cdot):\liealg_{\xi}^{c} \times \liealg_{\xi}^{c} \rightarrow \liealg_{\xi}^{c}$ defined by $2 g_{\xi,o}(U(e_i,e_j),e_k) = g_{\xi,o}([e_k,e_i]_{\liealg_{\xi}^{c}},e_j) + g_{\xi,o}(e_i,[e_k,e_j]_{\liealg_{\xi}^{c}})$ (Here $X_{\liealg_{\xi}^{c}}$ refers to the component of $X \in \liealg$ in $\liealg_{\xi}^{c}$ with respect to the direct sum decomposition $\liealg = \liealg_{\xi} \oplus \liealg_{\xi}^{c}$ ).
	
	Therefore, we use the formula $\Delta f = \sum_{i,j}g^{ij}\big(\dd{x_i} (\dd{x_j} f) - (\nabla_{\dd{x_i}} \dd{x_j}) (f) \big)$ in local coordinates to compute $\Delta f_{w,v}$ at the point $g \in G/G_{\xi}$. We have 
	\begin{align*}
	\dd{x_i} (\dd{x_j} f) &= \pdpd{}{x_i}{x_j} (f(g\cdot\exp(x_ie_i + x_j e_j))) = \pdpd{}{x_i}{x_j} \big\langle w ,\rho(g) \big( \rho( \exp(x_ie_i + x_j e_j)) (v) \big)\big\rangle \\
	& = \big\langle w, \rho(g) \big(d\rho(e_i) (d\rho(e_j)(v)) - \half d\rho([e_i,e_j])(v) \big) \big\rangle.
	\end{align*}
	Using eariler calculation from \eqref{eqn:lemma_calcuation} we get $\big((\nabla_{\dd{x_i}}\dd{x_j})f \big)|_{g} = \langle w, \rho(g) \big(d\rho(U(e_i,e_j))(v)\big) \rangle$, we obtain 
	\begin{align*}
	2\sum_{i,j}a^{ij} g_{\xi,o}(U(e_i,e_j),e_k) =&\sum_{i,j} a^{ij} g_{\xi,o}([e_k,e_i]_{\liealg_{\xi}^{c}},e_j) +\sum_{i,j} a^{ij} g_{\xi,o}(e_i,[e_k,e_j]_{\liealg_{\xi}^{c}}) \\
	=&2 \trace(\text{ad}(e_k)|_{\liealg_{\xi}^{c}}) = 0.
	\end{align*}
	Combining with the fact that $a^{ij}[e_i,e_j] = 0$ we obtain the desired identity.
	
\end{proof}

With the Karabegov moment map $\tilde{\moment}_{\xi} : X \rightarrow \liealg^*$, we have a natural embedding $\tilde{\moment}_{\xi}^* : \liealg \rightarrow \liealg \otimes (\liealg^*)^{G_{\xi}}  \hookrightarrow L^2(X)$ as a left $G$-submodule which is given by 
\begin{equation}\label{eqn:twisted_moment_map_explicit}
\tilde{\moment}_{\xi}(g) = \langle w,  \adjact^*(g) ( \xi + 2\delta_{\xi})  \rangle. 
\end{equation} 
We have the following theorem saying geometric quantization $\gqrep :\liealg \rightarrow  \text{End}(H^0(X,L_{\xi}))$ is given by composition of the Toeplitz quantization with the Karabegov moment map for coadjoint orbits.

\begin{theorem}\label{thm:coadjoint_orbit}
	For an integral coadjoint orbit $(X_{\xi} = G/G_{\xi}, \omega_{\xi},J,L_{\xi})$ with pre-quantium line bundle equipped with the natural Hamiltonian $G$-symmetry by left-translation, we have
	$$
	\gqrep = -i T \circ \tilde{\moment}_{\xi}^*.
	$$
\end{theorem}

\begin{proof}
	Making use of the Lemma \ref{lem:explicit_laplacian}, we compute $\Omega \cdot \xi$, or equivalently $\Omega \cdot \xi^{\vee}$ via the identification induced by the Killing form. For each $\alpha \in \rootspace_{+}$, we choose triple $iH_{\alpha}, X_{\alpha}, Y_{\alpha}$ lying in $\liealg$ such that 
	\begin{align*}
	[iH_\alpha,X_{\alpha}] & = 2 Y_\alpha\\
	[iH_{\alpha},Y_{-\alpha}] & = -2 X_{\alpha}\\
	[X_{\alpha},Y_{-\alpha}] &=  iH_{\alpha},
	\end{align*}
	where $H_{\alpha} $ is the coroot as in \S \ref{sec:root_space}. Therefore $\{X_{\alpha}, Y_{\alpha}\}_{\alpha \in \rootspace_{\xi}^{c}}$ form a basis for $\liealg_{\xi}^{c}$. From the discussion in \cite[Chapter 3 \S 7]{arvanitogeorgos2003introduction} about the Riemannian metric, we notice that this is an orthogonal basis with $\|X_{\alpha}\|^2 = \|Y_{\alpha}\|^2 =\xi(iH_{\alpha})$. Therefore we have
	\begin{align*}
	\Omega (\xi^{\vee}) &= \sum_{\alpha \in \rootspace_{\xi}^{c}} \xi(iH_{\alpha})^{-1} \big([X_{\alpha},[X_{\alpha}, \xi^{\vee}]]+ [Y_{\alpha},[Y_{\alpha}, \xi^{\vee}]] \big)\\
	& = \sum_{\alpha \in \rootspace_{\xi}^{c}} \xi(iH_{\alpha})^{-1} \alpha(\xi^{\vee}) \big([X_{\alpha},-Y_{\alpha}]+ [Y_{\alpha},X_{\alpha}] \big)\\
	& =  \sum_{\alpha \in \rootspace_{\xi}^{c}} \frac{\kappa(\alpha,\alpha)}{2i} (-2iH_{\alpha}) = -2 \sum_{\alpha \in \rootspace_{\xi}^{c}} \alpha^{\vee}. 
	\end{align*}
	Therefore we have $\Omega (\xi) = -4 \delta_{\xi}$. 
	
	By writing $\langle w , \tilde{\moment}_{\xi} \rangle =   f_{w,\xi} + 2 f_{w,\delta_{\xi}}$ for any $w \in \liealg$, we therefore have
	$$
	-iT\big( \langle w , \tilde{\moment}_{\xi} \rangle \big) = -iT \big(  f_{w,\xi} + 2f_{w,\delta_{\xi}}  \big) = -iT \big( f_{w,\xi} - \half \Delta(f_{w,\xi}) \big)  =Q_{f_{w,\xi}},
	$$
	using Lemma \ref{lem:tuynman_formula}, where $Q_{f} = \Pi \circ ( \nabla_{X_f} - i   f) $ as in \S \ref{sec:quantization}. Making use of the fact that $ [Q_{f_{w,\xi}}, Q_{f_{v,\xi}}] =  Q_{\{f_{w,\xi},f_{v,\xi}\}}$ again as in the proof of Theorem \ref{thm:general_kahler_einstein_thm}, and the fact that $\{f_{w,\xi},f_{v,\xi}\} = f_{[w,v],\xi}$ we obtain the identity $ [Q_{f_{w,\xi}}, Q_{f_{v,\xi}}]  = Q_{f_{[w,v],\xi}}$. As a result, one obtains the Lie algebra representation $\gqrep$ via $w \mapsto Q_{f_{w,\xi}}$. 
	\end{proof}

\section{Quantization via Toeplitz operators}\label{sec:quantization}

In this section, we given an explanation of how $\tilde{\mu}$ is related to quantum moment map in Fedosov deformation quantization.

\subsection{Berezin-Toeplitz quantization with Hamiltonian $G$-action}
\

\subsubsection{Karabegov moment map}\label{subsection: quantum-moment-map}
In this subsection, we explain the definition {\em Karabegov moment map} as a variation of quantum moment map in deformation quantization. Recall that a {\em quantum moment map} is a Lie algebra homomorphism $\mu_\hbar^*:\mathfrak{g}\rightarrow \left(C^\infty(X)[[\hbar]],\frac{i}{\hslash}[\cdot,\cdot]_\star\right)$ , such that for any $v\in\mathfrak{g}$:
$$
v(f)=\frac{i}{\hslash}[\mu_\hbar^*(v),f]_\star. 
$$

\begin{remark}\label{remark: normalization-factor}
	The normalization factor $\frac{i}{\hbar}$ guarantees that $\lim_{\hbar \rightarrow 0} \frac{i}{\hbar}[\cdot,\cdot]_{\star} = \{\cdot,\cdot\}$. With this normalization,  $-i\hbar\cdot\mu_\hslash^*$ is a Lie algebra homomorphism if we take the bracket $[\cdot,\cdot]_\star$ on $C^\infty(X)[[\hbar]]$. In particular, $\mu_\hslash|_{\hslash=0} = \mu$. This also explains the coefficient $-i$ in the commutative diagram \eqref{equation: commutative-diagram} (There we set $\hbar=1$). 
	
	\begin{comment}
 It is known that the Berezin-Toeplitz star product $\star$ satisfies that
$$
 \frac{i}{\hbar}\left(f\star g-g\star f\right)=\{f,g\}+O(\hslash).
 $$
Thus the normalization $\frac{i}{\hbar}[\cdot,\cdot]_\star$ on $C^\infty(X)[[\hbar]]$ is a quantization of the Poisson bracket $\{\cdot, \cdot\}$. Or equivalently, $-i\hbar\cdot\mu_\hslash^*$ is a Lie algebra homomorphism if we take the bracket $[\cdot,\cdot]_\star$ on $C^\infty(X)[[\hbar]]$. In particular, $\mu_\hslash^*|_{\hslash=0}$ is a classical moment map. This also explains the coefficient $-i$ in the commutative diagram \eqref{equation: commutative-diagram} (There we set $\hbar=1$). 
\end{comment}
\end{remark}

It is shown in \cites{Gutt-Rawnsley,Muller-Neumaier} that when a deformation quantization is induced by a Fedosov connection $D_F=\nabla^{\mathcal{W}}+\frac{1}{\hslash}[\gamma,-]_\star$ satisfying 
\begin{equation}\label{equation: Fedosov-equation}
\nabla^{\mathcal{W}}\gamma+\frac{1}{\hslash}\gamma\star\gamma=-\omega+\hslash\alpha_1+\hslash^2\alpha_2+\cdots=-\omega+\Omega.
\end{equation}
Then a quantum moment map must satisfy the following equation
\begin{equation}\label{equation: condition-quantum-moment-map}
 \iota_v(\omega-\Omega)=d\mu_\hbar^*(v).
\end{equation}

%Although the proof for equation \eqref{equation: condition-quantum-moment-map} is for the Weyl bundle with fiberwise Moyal product, it also works for the situation with the fiberwise Wick product. 

A deformation quantization on a K\"ahler manifold $X$ is called of Wick type (also known as separation of variables) if all the bi-differential operators $C_l(f,g)$ take holomorphic and anti-holomorphic derivatives of $f$ and $g$ respectively. It is shown in \cite{Karabegov96} that to every Wick type star product, there is an associated closed formal $(1,1)$-form $-\frac{1}{\hslash} \omega + \alpha_1 + \alpha_2 \hslash + \alpha_3 \hslash^2 + \cdots$ known as the {\em Karabegov form}, which gives rise to a one-one correspondence.

In \cite{CLL}, it is shown that there is a family of Fedosov connections induced from quantization of $L_\infty$ structure on K\"ahler manifolds, such that the formal closed $(1,1)$-form $-\omega+\Omega$ in equation \eqref{equation: Fedosov-equation} is exactly the Karabegov form of the associated star product. The Berezin-Toeplitz quantization is a Wick type deformation quantization whose Karabegov form is $-\frac{1}{\hslash} \omega-i\cdot\text{Ric}_X$. Since this formal $(1,1)$-form only has two terms in the $\hslash$ power expansion, we can turn the formal variable to any complex number without the convergence issue. In particular, equation \eqref{equation: condition-quantum-moment-map} with $1/\hslash=1$ gives the definition of Karabegov map (Definition \ref{def:karabegov_moment_map}). 

The formal variable $1/\hslash$ plays the role of the tensor power $m$ of the prequantum line bundle $L$. It is natural to define a family of Karabegov moment maps $\tilde{\mu}_m$ associated to the form $\omega+\frac{i}{m}\cdot\text{Ric}_X,\ m\in\mathbb{N}$. By considering $L^{\otimes m}$ and let $\mathcal{H}_m:=H_{L^2}^0(X,L^{\otimes m})$ when $X$ being coadjoint orbit or compact K\"ahler-Einstein, we have following commutative diagrams: 
\begin{equation*}
\xymatrix
{ & C^\infty(X) \ar[dr]^{T}  & \\
	\mathfrak{g} \ar[ur]^{-i\cdot\tilde{\mu}_m^*}\hspace{2mm}  \ar[rr]_{\beta_m} & & gl(\mathcal{H}_m)
}
\end{equation*}
in these two cases.

For general K\"ahler manifolds, the above diagram only commutes in an asymptotic sense:
$$
\big{|}\big{|}\beta+iT\circ\tilde{\mu}_m^*\big{|}\big{|}=O(m^{-\infty}).
$$
Here $\big{|}\big{|}\cdot \big{|}\big{|}$ denotes the operator norm. Equivalently, for any $k>0$, there exists a $C_k>0$, such that $\big{|}\big{|}\beta+iT\circ\tilde{\mu}_m^*\big{|}\big{|}\leq C_km^{-k}$.

\end{document}